\documentclass[11pt]{article}

\usepackage{amsmath,amssymb,amsthm,amsfonts,latexsym,mathtools,mathdots,bbm,graphicx,float}

\usepackage{braket,caption,subcaption,ellipsis,xcolor,textcomp,thmtools,thm-restate,cmap,mleftright}
\usepackage[colorlinks,linkcolor=purple,citecolor=blue]{hyperref}
\usepackage[nameinlink]{cleveref}
\usepackage[letterpaper, margin=0.9in]{geometry}
\usepackage{epigraph}
\usepackage{tikz} 

\usepackage{rubikcube,rubikrotation,rubikpatterns}
\usepackage[textsize=tiny]{todonotes}
\usepackage{newpxtext}
\usetikzlibrary{patterns,snakes}
 \usepackage{mathabx}
\usepackage{pgfplots}

\usepackage{booktabs}
\usepackage{wrapfig}

\usepackage[T1]{fontenc}

\let\Pr\relax
\DeclareMathOperator*{\Pr}{\mathbf{Pr}}

\newcommand\customBinom[2]{\mleft(\knds\genfrac..{0pt}{}{#1}{#2}\knds\mright)}
\newcommand{\knds}{\kern-\nulldelimiterspace}

\pgfplotsset{compat=1.18}

\newtheorem{theorem}{Theorem}
\newtheorem{lemma}{Lemma}
\newtheorem{definition}{Definition}
\newtheorem{proposition}{Proposition}

\title{A Demigod's Number for the Rubik's Cube}

\author{Arturo Merino$^\dagger$ \and Bernardo Subercaseaux$^\ddagger$}

\date{$^\dagger$Universidad de O'Higgins, Rancagua, Chile\\ 
$^\ddagger$Carnegie Mellon University, Pittsburgh, USA}

\begin{document}
\maketitle
\begin{abstract}
    It is by now well-known that any state of the $3\times 3 \times 3$ Rubik's Cube can be solved in at most 20 moves, a result often referred to as \emph{``God's Number''}. However, this result took Rokicki et al.~around 35 CPU years to prove and is therefore very challenging to reproduce. 
    We provide a novel approach to obtain a worse bound of 36 moves with high confidence, but that offers two main advantages: (i) it is easy to understrand, reproduce, and verify,  and (ii) our main idea generalizes to bounding the diameter of other vertex-transitive graphs by at most twice its true value, hence the name \emph{``demigod number''}. 
    Our approach is based on the fact that, for vertex-transitive graphs, the average distance between vertices is at most half the diameter, and by sampling uniformly random states and using a modern solver to obtain upper bounds on their distance, a standard concentration bound allows us to confidently state that the average distance is around $18.32 \pm 0.1$, from where the diameter is at most $36$. 
\end{abstract}


\section{Introduction}\label{sec:intro}

The $3\times 3 \times 3$ Rubik's Cube, illustrated in~\Cref{fig:rubik}, is arguably one of the most iconic puzzles ever created, and one of the best-selling toys of all time; its beautiful balance of simplicity (it only has 6 faces, with 54 colored stickers) and complexity (it has over $4.3 \times 10^{19}$ possible states) has captured the attention of millions of people around the world since its invention in 1974. 
Naturally, such a combinatorially rich puzzle has raised a variety of interesting mathematical questions, with the most famous one being:
\begin{center}
    \emph{What is the minimum number of moves required to solve the Rubik's Cube from any starting position?}
\end{center}
To make this question precise, we consider the \emph{half-turn metric}, in which turning either of the 6 faces of the cube by any amount (i.e., $90^\circ$, $180^\circ$, or $270^\circ$) counts as a single move. 
After a series of incremental improvements detailed in~\Cref{tab:bounds}, Rokicki et al.~\cite{rokickiDiameterRubikCube2014} proved that $20$ moves are always enough to solve any Rubik's cube in the half-turn metric, a result often referred to as \emph{``God's Number''}. 
This result, however, was obtained by a mixture of mathematical ideas and extensive computation, taking around 35 CPU years. As a result, verifying the correctness of the so-called God's Number is extremely challenging, and the required computations have likely never been reproduced independently. While the result is widely believed to be true, our goal is to provide an alternative approach (providing a weaker bound) that can be easily understood and reproduced in e.g., a classroom setting.

\begin{table}[h]
    \caption{Historical bounds on the maximum number of moves required to solve the Rubik's Cube~\cite{cube20GodapossNumber}. }\label{tab:bounds}
    \centering
    \begin{tabular}{@{}rcc@{}}
    \toprule
    Year      & Lower bound & Upper bound \\ \midrule
    1981      & 18          & 52          \\
     1990  & 18          & 42          \\
 1992       & 18          & 39          \\
 1992       & 18          & 37          \\
    1995   & 18          & 29          \\
     1995   & 20          & 29          \\
     2005  & 20          & 28          \\
     2006     & 20          & 27          \\
 2007       & 20          & 26          \\
 2008     & 20          & 25          \\
     2008     & 20          & 23          \\
     2008    & 20          & 22          \\
     2010      & 20          & 20          \\
    \bottomrule
    \end{tabular}
\end{table}

\subsection{Summary}
Our approach is based on the following observation:
in any vertex-transitive graph (i.e., a graph where every vertex \emph{``looks the same''}), the diameter $D$ (i.e., the maximum distance between any two vertices) is at most twice the mean distance between vertices, $m$. This observation corresponds to a question posed by Alan Kaplan~\cite{illinoisProblemsGraph}, which was answered in the more general context of homogeneous compact metrics spaces by Herman and Pakianathan~\cite{HERMAN201597}. Our proof, however, is elementary and self-contained, and we believe it can be a nice addition to a first course on graph theory.

Applying the previous observation to the Cayley graph of the Rubik's cube (defined formally in~\Cref{sec:prelim}, this graph has the possible states of the cube as vertices and edges between pair of states that are one ``move'' away from each other), if we knew the average distance $\mu$ between states of the Rubik's Cube, we could bound the diameter of the Rubik's Cube by at most twice this value. While we cannot directly compute the exact value of $\mu$ for the Rubik's Cube\footnote{Recall that it has over $4 \cdot 10^{19}$ many states.}, we can provide a good estimate by sampling random pairs of states and computing an upper bound on their distance through standard Rubik's algorithms (i.e., Kociemba's \emph{Two-Phase Algorithm}~\cite{kociemba}). Each of these upper bounds for the distance between a pair of states is certified by a short sequence of moves, so one can trust the result without needing to trust the algorithm.
Through simple concentration bounds, we will argue that the empirical average of $\widehat{\mu} \approx 18.3189$ we obtain is \emph{overwhelmingly likely} to be a good estimate of the true mean distance $\mu$, and therefore an audience should reasonably trust that the diameter of the Rubik's Cube is at most $36$ (since it must be an integer). As a first step, we will prove the following theorem.
\begin{restatable}{theorem}{mytheorem}
    \label{thm:main-1}
    Given a state $s$ of the Rubik's Cube, let $d(s)$ be the distance from $s$ to the solved state. Let $S$ be a set of states of the Rubik's cube sampled uniformly at random, and let $\widehat{\mu}_S = \frac{1}{|S|}\sum_{s \in S} d(s)$ be the random variable corresponding to the average distance between states in $S$ and the solved state.
    Then, if $D$ denotes the diameter of the Rubik's Cube, we have 
    \[
        \Pr_S\left[ D \geq 2\widehat{\mu}_S +  0.36  \right] < 2\exp\left(\frac{-|S|}{1\,541\,939}\right).
    \]
\end{restatable}
Note immediately that~\Cref{thm:main-1} implies that if $D > 36$ (and thus $D \geq 37$ since the diameter is an integer), the probability of obtaining $\widehat{\mu}_S \approx 18.318$ for $|S| = 10^7$ is less than $10^{-7}$. However, this experience can be observed repeatedly, which is therefore a great degree of probabilistic evidence for $D \leq 36$.
Unfortunately,~\Cref{thm:main-1} is somewhat computationally expensive, as it requires the number of samples $|S|$ to be around $10$ million if we desire a probability of error under $10^{-7}$. As we show in~\Cref{sec:demigod}, this takes roughly $100$ hours of computation. In order to see how to reduce the required computation, let us briefly discuss how~\Cref{thm:main-1} is obtained.
To obtain an upper bound on $D$, we can leverage the $D < 2\mu$ observation and look for an upper bound on $\mu$, which we can do with a probabilistic guarantee by considering an empirical average $\mu_{S}$. However, a priori it could be that the true average $\mu$ is much larger than our empirical estimate $\mu_S$ due to a small number of states that are very far and we are not likely to sample randomly; a ``long tail'' phenomenon. Our solution to this problem is using a 
 \emph{``Human's number''}, which is an unconditional modest upper bound on $D$. For instance, the so-called  \emph{``beginner's method''} suffices to obtain an upper bound of $205$ moves.
\begin{lemma}[Human's Number]\label{lem:human}
    Any position of the Rubik's cube can be solved in at most 205 moves.\footnote{A proof sketch is provided in the appendix. In general, we expect anyone familiar with the beginner's method to find this bound trivial, since it is far from being tight. }
\end{lemma}
Thus, in general, our work can be interpreted as a method for transforming a \emph{``Human's number''} into a \emph{``Demigod number''} that is easy to trust and at most twice the real~\emph{``God's number''}.
Now, going back to the problem of how to reduce the number of samples, the $1\,541\,939$ constant in~\Cref{thm:main-1} is a consequence of the constant 205 in~\Cref{lem:human}.
We can improve on this by showing first that in a large percetange of the cases, we can use an upper bound of $20$ instead of $205$. Indeed, let us say that a state is \emph{``far from being solved''} if it requires strictly more than $20$ moves to be solved. Naturally, the ``God's number'' result corresponds to the inexistence of any state that is ``far from being solved'', but proving this requires a great computational effort. Instead, we use a much simpler argument: if the proportion of states that are ``far from being solved'' were to be bigger than, say, $0.03\%$, then we would certainly expect to see a state that is ``far from being solved'' after $50\,000$ random samples.
Yet, experimentally we do not see any such state after $500\,000$ samples, which makes the possibility of more than $0.03\%$ of total states being ``far from being solved'' extremely slim. This way, we can separate $\mu$ into:
\[
    \left[p_\text{far} \cdot \sum_{s \text { far from being solved}} d(s)\right] + \left[(1-p_\text{far})\cdot \sum_{s \text { not far from being solved}} d(s) \right]
\]

Finally, we believe that our approach can be used as an example to motivate philosophical conversations about \emph{plausibility}, a notion explored at large by George Polya in his book \emph{``Mathematics and Plausible Reasoning''}~\cite{polyaMathematicsPlausibleReasoning1954}. In a nutshell, the idea is that while certain types of reasoning do not provide full proof of a statement, they can provide a high degree of confidence in its truth~\cite{borweinMathematicsExperimentPlausible2011}. This is the case, for example, with probabilistic primality tests, which may allow us to confidently assert that a number with hundreds of millions of digits is prime, without providing a proof in the traditional sense of the word.

\begin{figure}[t]
    \begin{subfigure}{0.25\textwidth}
        \centering
        \RubikCubeSolved%
        \begin{tikzpicture}[scale=0.5]
            \DrawRubikCubeRU
        \end{tikzpicture}
        \caption{Solved state}
    \end{subfigure}
    \begin{subfigure}{0.45\textwidth}
        \centering
        \RubikCubeSolved%
        \begin{tikzpicture}[scale=0.5]
            \DrawRubikCubeF
            \node (U) at (1.5, 4.5) [black]{\small\textsf{U}};
            \node (D) at (1.5, -1.5) [black]{\small\textsf{D}};
            \node (L) at (-1.5, 1.5) [white]{\small\textsf{L}};
            \node (R) at (4.5, 1.5) [white]{\small\textsf{R}};
            \node (F) at (1.5, 1.5) [white]{\small\textsf{F}};
            \node (B) at (7.5, 1.5) [white]{\small\textsf{B}};
        \end{tikzpicture}
        \caption{Flat view of the solved state, letters show how we will refer to faces (see ~\Cref{sec:beginner})}
    \end{subfigure}
    \begin{subfigure}{0.25\textwidth}
        \centering
        \includegraphics{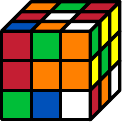}
        \caption{A scrambled state}
    \end{subfigure}
      
    \caption{Illustration of the $3\times 3 \times 3$ Rubik's Cube.}\label{fig:rubik}
\end{figure}
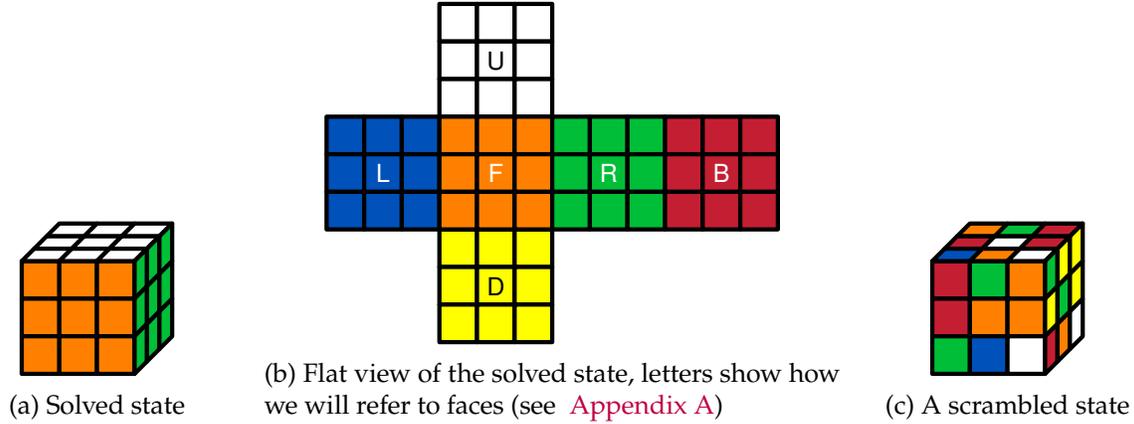

\section{Preliminaries}\label{sec:prelim}
Let us introduce the notation and definitions required to work over the Rubik's cube mathematically. We will see the Rubik's cube as a group, a standard idea (see e.g.,~\cite{chen2004group,Bandelow1982}) that we make explicit nonetheless to make our work as self-contained as possible.

Let $S_n$ denote the group of permutations of $n$ elements under the composition operation, denoted by $\circ$. For example, $(2, 3, 1) \in S_3$ is the permutation defined by: 
\[
 1 \mapsto 2,\; 2 \mapsto 3,\; 3 \mapsto 1 = \begin{pmatrix} 1 & 2 & 3 \\ 2 & 3 & 1 \end{pmatrix},
 \] 
 and 
\[
(2, 3, 1) \circ (1, 3, 2) = (2, 1, 3).
\] 
We identify each state of the Rubik's cube with a permutation of its stickers, of which there are $6 \cdot 9 = 54$ (9 per each of the 6 faces), and as moving faces in the Rubik's cube results in a permutation of its stickers, the Rubik's cube can be seen as a subgroup of $S_{54}$, as illustrated in~\Cref{fig:cube-group}.\footnote{More in general, a classic theorem of Cayley states that any group is isomorphic to a subgroup of a symmetric group~\cite[Section 1.3]{Jacobson2009-ak}.} 
Furthermore, as no moves affect the relative position of the center stickers ($5, 14, 23, 32, 41, 50$), we can see the Rubik's cube as a subgroup of $S_{48}$.
To complete the definition of the Rubik's cube group, we need to define the ``moves'', which correspond to the following permutations (omitting the values that are not affected by the move):
\setcounter{MaxMatrixCols}{20}
\[
    \text{\rr{R}} = \begin{pmatrix}
        3 & 6 & 9 & 21 & 24 & 27 & 28 & 29 & 30 & 31 & 33 & 34 & 35 & 36 & 37 & 40 & 43 & 48 & 51 & 54\\
       43 & 40 & 37 & 3 & 6 & 9 & 30 & 33 & 36 & 29 & 35 & 28 & 31 & 34 & 54 & 51 & 48 & 21 & 24 & 27
    \end{pmatrix},
\]
\[
 \text{\rr{L}} = \begin{pmatrix}
    1 & 4 & 7 & 10 & 11 & 12 & 13 & 15 & 16 & 17 & 18 & 19 & 22 & 25 & 39 & 42 & 45 & 46 & 49 & 52
     \\
    19 & 22 & 25 & 12 & 15 & 18 & 11 & 17 & 10 & 13 & 16 & 46 & 49 & 52 & 7 & 4& 1 & 45 & 42 & 39
    \end{pmatrix},
\]
\[
    \text{\rr{U}} = \begin{pmatrix}
        1 & 2 & 3 & 4 & 6 & 7 & 8 & 9 & 10 & 11 & 12 & 19 & 20 & 21 & 28 & 29 & 30 & 37 & 38 & 39\\
        3 & 6 & 9 & 2 & 8 & 1 & 4 & 7 & 37 & 38 & 39 & 10 & 11 & 12 & 19 & 20 & 21 & 28 & 29 & 30
        \end{pmatrix},
\]
and similarly, \rr{D}, \rr{F}, and \rr{B} can be deduced from~\Cref{fig:cube-group}. The moves \rr{Rp}, \rr{Lp}, \rr{Up}, etc., correspond to the inverse of the moves \rr{R}, \rr{L}, \rr{U}, etc., respectively.
A sequence of moves is simply a composition of moves, omitting the composition symbol, e.g., \rr{R} \rr{U}2 corresponds to the permutation $\text{\rr{R}} \circ \text{\rr{U}} \circ \text{\rr{U}}$.

\begin{figure}
    \centering
    \begin{tikzpicture}[scale=0.6]
        \RubikCubeSolved
        \DrawRubikCubeF
        \node (1) at (1 - 0.5, 6 -0.5) [black]{$1$};
        \node (2) at (2 - 0.5, 6 -0.5) [black]{$2$};
        \node (3) at (3 - 0.5, 6 -0.5) [black]{$3$};
        \node (4) at (1 - 0.5, 5 -0.5) [black]{$4$};
        \node (5) at (2 - 0.5, 5 -0.5) [black]{$5$};
        \node (6) at (3 - 0.5, 5 -0.5) [black]{$6$};
        \node (7) at (1 - 0.5, 4 -0.5) [black]{$7$};
        \node (8) at (2 - 0.5, 4 -0.5) [black]{$8$};
        \node (9) at (3 - 0.5, 4 -0.5) [black]{$9$};

        \node (10) at (1 - 0.5-3, 6 -0.5 - 3) [white]{$10$};
        \node (11) at (2 - 0.5-3, 6 -0.5 - 3) [white]{$11$};
        \node (12) at (3 - 0.5-3, 6 -0.5 - 3) [white]{$12$};
        \node (13) at (1 - 0.5-3, 5 -0.5 - 3) [white]{$13$};
        \node (14) at (2 - 0.5-3, 5 -0.5 - 3) [white]{$14$};
        \node (15) at (3 - 0.5-3, 5 -0.5 - 3) [white]{$15$};
        \node (16) at (1 - 0.5-3, 4 -0.5 - 3) [white]{$16$};
        \node (17) at (2 - 0.5-3, 4 -0.5 - 3) [white]{$17$};
        \node (18) at (3 - 0.5-3, 4 -0.5 - 3) [white]{$18$};

        \node (19) at (1 - 0.5, 6 -0.5 - 3) [white]{$19$};
        \node (20) at (2 - 0.5, 6 -0.5 - 3) [white]{$20$};
        \node (21) at (3 - 0.5, 6 -0.5 - 3) [white]{$21$};
        \node (22) at (1 - 0.5, 5 -0.5 - 3) [white]{$22$};
        \node (23) at (2 - 0.5, 5 -0.5 - 3) [white]{$23$};
        \node (24) at (3 - 0.5, 5 -0.5 - 3) [white]{$24$};
        \node (25) at (1 - 0.5, 4 -0.5 - 3) [white]{$25$};
        \node (26) at (2 - 0.5, 4 -0.5 - 3) [white]{$26$};
        \node (27) at (3 - 0.5, 4 -0.5 - 3) [white]{$27$};

        \node (28) at (1 - 0.5+3, 6 -0.5 - 3) [white]{$28$};
        \node (29) at (2 - 0.5+3, 6 -0.5 - 3) [white]{$29$};
        \node (30) at (3 - 0.5+3, 6 -0.5 - 3) [white]{$30$};
        \node (31) at (1 - 0.5+3, 5 -0.5 - 3) [white]{$31$};
        \node (32) at (2 - 0.5+3, 5 -0.5 - 3) [white]{$32$};
        \node (33) at (3 - 0.5+3, 5 -0.5 - 3) [white]{$33$};
        \node (34) at (1 - 0.5+3, 4 -0.5 - 3) [white]{$34$};
        \node (35) at (2 - 0.5+3, 4 -0.5 - 3) [white]{$35$};
        \node (36) at (3 - 0.5+3, 4 -0.5 - 3) [white]{$36$};

        \node (37) at (1 - 0.5+6, 6 -0.5 - 3) [white]{$37$};
        \node (38) at (2 - 0.5+6, 6 -0.5 - 3) [white]{$38$};
        \node (39) at (3 - 0.5+6, 6 -0.5 - 3) [white]{$39$};
        \node (40) at (1 - 0.5+6, 5 -0.5 - 3) [white]{$40$};
        \node (41) at (2 - 0.5+6, 5 -0.5 - 3) [white]{$41$};
        \node (42) at (3 - 0.5+6, 5 -0.5 - 3) [white]{$42$};
        \node (43) at (1 - 0.5+6, 4 -0.5 - 3) [white]{$43$};
        \node (44) at (2 - 0.5+6, 4 -0.5 - 3) [white]{$44$};
        \node (45) at (3 - 0.5+6, 4 -0.5 - 3) [white]{$45$};

        \node (46) at (1 - 0.5+0, 6 -0.5 - 6) [black]{$46$};
        \node (47) at (2 - 0.5+0, 6 -0.5 - 6) [black]{$47$};
        \node (48) at (3 - 0.5+0, 6 -0.5 - 6) [black]{$48$};
        \node (49) at (1 - 0.5+0, 5 -0.5 - 6) [black]{$49$};
        \node (50) at (2 - 0.5+0, 5 -0.5 - 6) [black]{$50$};
        \node (51) at (3 - 0.5+0, 5 -0.5 - 6) [black]{$51$};
        \node (52) at (1 - 0.5+0, 4 -0.5 - 6) [black]{$52$};
        \node (53) at (2 - 0.5+0, 4 -0.5 - 6) [black]{$53$};
        \node (54) at (3 - 0.5+0, 4 -0.5 - 6) [black]{$54$};
    \end{tikzpicture}
    \caption{Illustration of the Rubik's cube as a subgroup of $S_{54}$ (or even $S_{48}$ due to the centers being static).}\label{fig:cube-group}
\end{figure}

Let $\mathcal{R}$, the Rubik's cube group, be the subgroup of $S_{54}$ generated by the moves, i.e.,
\[
    \mathcal{R} = \langle \text{\rr{R}}, \text{\rr{L}}, \text{\rr{U}}, \text{\rr{D}}, \text{\rr{F}}, \text{\rr{B}} \rangle,
\]
which is well-defined since $\text{\rr{M}}^3 = \text{\rr{Mp}} = \text{\rr{M}}^{-1}$ for every move $\text{\rr{M}} \in \{\text{\rr{R}}, \text{\rr{L}}, \text{\rr{U}}, \text{\rr{D}}, \text{\rr{F}}, \text{\rr{B}}  \}$.
As described with words above, $\mathcal{R}$ is clearly isomorphic to a subgroup of $S_{48}$ due to the centers of each face being unaffected by the generators. 
Finally, note that this group perspective blurs the difference between sequences of moves and states of the cube; every move corresponds to a state of the cube (the result of applying the move to the solved state), and every state reachable from the solved state by moves corresponds to an equivalence class of all move sequences reaching that state.

\begin{definition}[Cayley Graph]
    Given a group $G = (A, \star)$ and a set of generators $S \subseteq A$, the \emph{Cayley graph} $G_S$ is the graph whose vertex set is $A$ and where two vertices $u, v \in G$ are adjacent if there exists a generator $s \in S$ such that $u \star s = v$.
\end{definition}

We will use $G_\mathcal{R}$ to denote the Cayley graph of the Rubik's cube group under the following set of generators:
 \[
    S_\mathcal{R} := \{\text{\rr{R}}, \text{\rr{Rp}}, \text{\rr{R}2}, \text{\rr{L}}, \text{\rr{Lp}}, \text{\rr{L}2}, \text{\rr{U}}, \text{\rr{Up}}, \text{\rr{U}2}, \text{\rr{D}}, \text{\rr{Dp}}, \text{\rr{D}2}, \text{\rr{F}}, \text{\rr{Fp}}, \text{\rr{F}2}, \text{\rr{B}}, \text{\rr{Bp}}, \text{\rr{B}2}\}.
\]
Note that the reason we include e.g., $\text{\rr{Rp}}$ and $\text{\rr{R}2}$ in $S_\mathcal{R}$ is that we want to count those as single moves of the cube. Counting $\text{\rr{R}}2$ as $2$ moves leads to a different metric, usually called the \emph{Quarter-turn Metric}, where \emph{God's number} is 26~\cite{cube26}.

\begin{definition}[Diameter]
  The \emph{diameter} $D$ of a graph $G = (V, E)$ is the maximum distance between any two vertices, where the distance between two vertices $u, v \in V$ is the length of the shortest path between them. That is,
    \[
        D = \max_{u,\, v\, \in \customBinom{V}{2}} d(u, v).
    \]
\end{definition}

In this language, ``God's number'' is simply the diameter of the Rubik's cube graph $G_\mathcal{R}$.

\begin{definition}[Mean distance]
    The \emph{mean distance} $\mu$ of a graph $G = (V, E)$ is the average distance between any two vertices, that is,
    \[
        \mu = \frac{1}{\customBinom{|V|}{2}}\sum_{u,\, v\, \in \customBinom{V}{2}} d(u, v).
    \]
\end{definition}

\begin{definition}[Graph automorphism]
    An \emph{automorphism} over a graph $G = (V, E)$ is a bijection $\varphi: V \to V$ such that for any two vertices $u, v \in V$, we have that $(u, v) \in E$ if and only if $(\varphi(u), \varphi(v)) \in E$.
\end{definition}

Using the previous definition repeatedly leads to the following trivial lemma.
\begin{lemma}\label{lemma:dist-auto}
    If $\varphi$ is an automorphism over a graph $G = (V, E)$, then for any two vertices $u, v \in V$, we have that $d(u, v) = d(\varphi(u), \varphi(v))$.
\end{lemma}

\begin{definition}[Vertex Transitivity]
    A graph $G = (V, E)$ is \emph{vertex-transitive} if for any two vertices $u, v \in V$, there exists an automorphism $\varphi$ such that $\varphi(u) = v$.
\end{definition}

We now state a folklore idea that will be key to our analysis of the Rubik's cube graph.

\begin{lemma}
    \label{lemma:cayley}
    Every Cayley graph is vertex-transitive, and in particular, the Rubik's cube graph $G_\mathcal{R}$ is vertex-transitive.
\end{lemma}
\begin{proof}
    Let $G = (A, \star)$ be a group and $S \subseteq A$ a set of generators. We must prove that there is an automorphism $\varphi$ such that $\varphi(u) = v$ for any two vertices $u, v \in A$.
    Let us define
    \[
        \varphi: A \to A, \quad \varphi(x) = v \star u^{-1} \star x.
    \]
   This definition directly implies $\varphi(u) = v$, and $\varphi$ is clearly bijective since
   \( 
  x \mapsto  u \star v^{-1} \star x
   \)
   is an inverse for $\varphi$.
   It remains to prove that for any two vertices $x, y \in A$, we have that $(x, y) \in E$ if and only if $(\varphi(x), \varphi(y)) \in E$. 
    Where, by definition of the Cayley graph, a pair of vertices $(a, b)$ is in $E$ if there exists a generator $s \in S$ such that $a \star s = b$, and equivalently, if $a^{-1} \star b \in S$.
  Now observe that
   \begin{align*}
    \varphi(x)^{-1} \star \varphi(y) &= \left(v \star u^{-1} \star x\right)^{-1} \star (v \star u^{-1} \star y)\\
    &= (x^{-1} \star u \star v^{-1}) \star (v \star u^{-1} \star y)\\
    &= (x^{-1} \star u) \star (v^{-1} \star v) \star  (u^{-1} \star y) \\
    &= x^{-1} \star (u^{-1} \star u) \star y = x^{-1} \star y,
   \end{align*}
   from where we conclude by noting that
   \begin{align*}
    (x, y) \in E &\iff x^{-1} \star y \in S\\
                &\iff \varphi(x)^{-1} \star \varphi(y)\in S\\
                &\iff (\varphi(x), \varphi(y)) \in E.\qedhere
   \end{align*}
\end{proof}

We conclude this section with a simple lemma stating that in a vertex-transitive graph, given that all nodes are ``essentially the same'', we can think of the mean distance as the average distance from a fixed node, instead of between all pairs.
\begin{lemma}
    Let $x$ be any vertex in a vertex-transitive graph $G$. Then we have 
     \(
         \mu = \frac{\sum_{v \in V} d(x, v)}{|V| - 1}.
     \)
     \label{lemma:dist-vertex-transitivity}
 \end{lemma}
 \begin{proof}
     First, note that by definition of mean distance, and using $d(u, u) = 0$, we have
     \[
         \mu = \frac{1}{\customBinom{|V|}{2}}\sum_{u,\, v\, \in \customBinom{V}{2}} d(u, v) = \frac{1}{|V| (|V|-1)}\sum_{u \in V} \sum_{v \in V} d(u, v).
     \]
     Because of vertex-transitivity, for any vertex $u \in V$, there exists an automorphism $\varphi_u$ such that $\varphi_u(u) = x$. Therefore, using~\Cref{lemma:dist-auto}, we have
     \[
         \mu = \frac{1}{|V| (|V|-1)}\sum_{u \in V} \sum_{v  \in V} d(u, v) = \frac{1}{|V| (|V|-1)}\sum_{u \in V} \sum_{v \in V}  d(x, \varphi_u(v))\\
     \]
     But as $\varphi_u : V \to V$ is a bijection for every $u$, we have 
     \(
         \sum_{v \in V} d(x, \varphi_u(v)) = \sum_{v \in V} d(x, v),
     \)   
     and thus 
     \[
         \mu =  \frac{1}{|V| (|V|-1)}\sum_{u \in V} \sum_{v \in V} d(x, v) = \frac{|V|}{|V| (|V|-1)}\sum_{v \in V} d(x, v) = \frac{\sum_{v \in V} d(x, v)}{|V| - 1}.\qedhere
     \]
 
 \end{proof}

\section{The Relationship Between Diameter and Mean Distance}\label{sec:results}
In this section, we explore the relationship between the diameter and the mean distance of a graph, and show that for vertex-transitive graphs the diameter is at most twice the mean distance. While this result is implied by~\cite{HERMAN201597}, we offer a more elementary exposition.

First, let us note that in arbitrary graphs, the diameter $D$ can be much larger than the mean distance $\mu$. 
\begin{proposition}[Folklore, cf.~\cite{wuCONJECTUREAVERAGEDISTANCE2011}]\label{lemma:clique}
    For every $n$, there are graphs on $n$ vertices such that $D/\mu = \Omega(n^{1/2})$.
\end{proposition}
\begin{proof}
 We can construct a graph $G$ by taking a clique on $n$ vertices and attaching to it a path on $n^{1/2}$ vertices, as illustrated in~\Cref{fig:clique}.

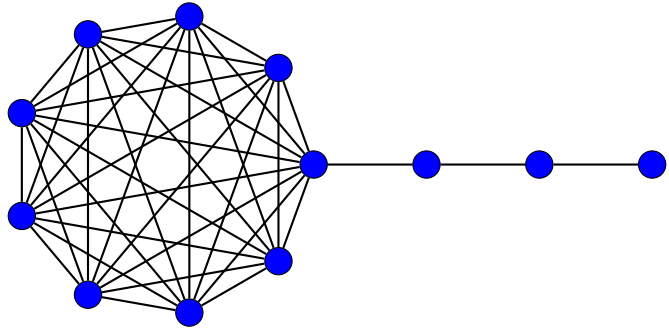
\begin{figure}[H]
\centering
\begin{tikzpicture}




    \node[draw, circle, fill=blue] (clique_0) at (2.0, 0.0) {};
\node[draw, circle, fill=blue] (clique_1) at (1.532088886237956, 1.2855752193730785) {};
\node[draw, circle, fill=blue] (clique_2) at (0.34729635533386083, 1.969615506024416) {};
\node[draw, circle, fill=blue] (clique_3) at (-0.9999999999999997, 1.7320508075688774) {};
\node[draw, circle, fill=blue] (clique_4) at (-1.8793852415718166, 0.6840402866513378) {};
\node[draw, circle, fill=blue] (clique_5) at (-1.8793852415718169, -0.6840402866513373) {};
\node[draw, circle, fill=blue] (clique_6) at (-1.0000000000000009, -1.7320508075688767) {};
\node[draw, circle, fill=blue] (clique_7) at (0.34729635533385994, -1.9696155060244163) {};
\node[draw, circle, fill=blue] (clique_8) at (1.5320888862379556, -1.2855752193730792) {};
\draw[thick] (clique_0) -- (clique_1);
\draw[thick] (clique_0) -- (clique_2);
\draw[thick] (clique_0) -- (clique_3);
\draw[thick] (clique_0) -- (clique_4);
\draw[thick] (clique_0) -- (clique_5);
\draw[thick] (clique_0) -- (clique_6);
\draw[thick] (clique_0) -- (clique_7);
\draw[thick] (clique_0) -- (clique_8);
\draw[thick] (clique_1) -- (clique_2);
\draw[thick] (clique_1) -- (clique_3);
\draw[thick] (clique_1) -- (clique_4);
\draw[thick] (clique_1) -- (clique_5);
\draw[thick] (clique_1) -- (clique_6);
\draw[thick] (clique_1) -- (clique_7);
\draw[thick] (clique_1) -- (clique_8);
\draw[thick] (clique_2) -- (clique_3);
\draw[thick] (clique_2) -- (clique_4);
\draw[thick] (clique_2) -- (clique_5);
\draw[thick] (clique_2) -- (clique_6);
\draw[thick] (clique_2) -- (clique_7);
\draw[thick] (clique_2) -- (clique_8);
\draw[thick] (clique_3) -- (clique_4);
\draw[thick] (clique_3) -- (clique_5);
\draw[thick] (clique_3) -- (clique_6);
\draw[thick] (clique_3) -- (clique_7);
\draw[thick] (clique_3) -- (clique_8);
\draw[thick] (clique_4) -- (clique_5);
\draw[thick] (clique_4) -- (clique_6);
\draw[thick] (clique_4) -- (clique_7);
\draw[thick] (clique_4) -- (clique_8);
\draw[thick] (clique_5) -- (clique_6);
\draw[thick] (clique_5) -- (clique_7);
\draw[thick] (clique_5) -- (clique_8);
\draw[thick] (clique_6) -- (clique_7);
\draw[thick] (clique_6) -- (clique_8);
\draw[thick] (clique_7) -- (clique_8);
\node[draw, circle, fill=blue] (path_0) at (3.5, 0) {};
\node[draw, circle, fill=blue] (path_1) at (5.0, 0) {};
\node[draw, circle, fill=blue] (path_2) at (6.5, 0) {};
\draw[thick] (clique_0) -- (path_0);
\draw[thick] (path_0) -- (path_1);
\draw[thick] (path_1) -- (path_2);
\end{tikzpicture}
\caption{A graph $G$ with $n = 9$, illustrating the proof of~\Cref{lemma:clique}.}\label{fig:clique}
\end{figure}

 The diameter of this graph is $D = n^{1/2} + 1$. Noting that $|V| = n + n^{1/2} = \Omega(n)$, the mean distance can be calculated as follows:
 \begin{align*}
    \mu &=\frac{1}{\customBinom{|V(G)|}{2}}\sum_{u, \, v \, \in \customBinom{V(G)}{2}} d(u, v)\\
     &= \frac{1}{\Omega(n^2)} \bigg( \underbrace{1 \cdot O(n^2)}_{\text{between clique vertices}} + \underbrace{O(n^{1/2}) \cdot O(n)}_{\text{between path vertices}} + \underbrace{O(n^{1/2}) \cdot n \cdot n^{1/2}}_{\text{clique-to-path}} \bigg)\\
     &= O(1). \qedhere
\end{align*}
\end{proof}
Furthermore, Wu et al. proved that this bound is asymptotically tight, meaning that $D/\mu = O(n^{1/2})$~\cite{wuCONJECTUREAVERAGEDISTANCE2011}.
It turns out, however, that such a gap between $D$ and $\mu$ is not possible in vertex-transitive graphs, where $D$ and $\mu$ are always a constant factor away.

\begin{theorem}
    For any vertex-transitive graph $G$ of diameter $D$ and mean distance $\mu$ we have $D < 2\mu$.
    \label{thm:d-m-vertex-transitive}
\end{theorem}
\begin{proof}
    Let $u, v$ be any pair of vertices such that $d(u, v) = D$, and use~\Cref{lemma:dist-vertex-transitivity} to write
    \begin{equation}\label{eq:vertex-trans}
    \mu = \frac{\sum_{x \in V} d(u, x)}{|V|-1} = \frac{\sum_{x \in V} d(v, x)}{|V|-1},
    \end{equation}
    from where
    \begin{align*}
    2\mu &= \frac{\sum_{x \in V} d(u, x)}{|V|-1} + \frac{\sum_{x \in V} d(v, x)}{|V|-1}\\
        &\geq \frac{\sum_{x \in V} d(u, v)}{|V| - 1} \tag{Triangle inequality}\\
        &= \frac{|V| \cdot D}{|V| - 1} > D.\qedhere
    \end{align*}
\end{proof}

We can see that this is tight by considering a cycle on $2n$ vertices, where $D = n$ and using~\Cref{lemma:dist-vertex-transitivity}, 
 
\[ 
\mu = \frac{1}{2n-1} \sum_{v \in V} d(u, v) = \frac{1}{2n-1} \left(2\left(\sum_{i=1}^{n-1} i\right)  + n\right) = \frac{2 \cdot \frac{n(n-1)}{2} + n}{2n-1} = \frac{n^2}{2n-1} = \frac{n}{2} + o(1).
\]
Similarly, for the hypercube graph $Q_n$, we have $D = n$ and 
\begin{align*}
    \mu &= \frac{1}{2^n - 1} \sum_{v \in V} d(u, v) = \frac{1}{2^n - 1} \sum_{k=1}^{n} \sum_{\substack{v \in V, \\ d(u, v) = k}} k \\
     &=  \frac{1}{2^n-1} \sum_{k=1}^{n} k \binom{n}{k} = \frac{n2^{n-1}}{2^n - 1} = \frac{n}{2} + o(1).
\end{align*}



\section{The Demigod Number for the Rubik's Cube}\label{sec:demigod}
\Cref{thm:d-m-vertex-transitive} allows us to translate upper bounds for the average distance into upper bounds for the diameter.
This is particularly useful, as the average distance is easier to certify with high confidence than the diameter. 
In order to provide an upper bound of the average distance we will follow the following strategy:
\begin{itemize}
    \item We will sample a large number (500\,000) of uniformly random states of the Rubik's cube.
    \item For each state, we use an efficient solver\footnote{We use \url{https://github.com/efrantar/rob-twophase} since it was the fastest solver we could find online.} to obtain an upper bound on the distance, which is certified by the move sequence that the solver outputs.
    \item We use a simple concentration bound to argue that the empirical average of the distances is a good estimate of the true average distance.
\end{itemize}

Indeed, let us prove~\Cref{thm:main-1} right away, after which we will explain the precise methodology used to efficiently obtain a diameter upper bound with high confidence.
Besides~\Cref{thm:fundamental-thm}, we need the following standard concentration inequality.
\begin{lemma}[Hoeffding's inequality]
    \label{lem:hoeffding}
    Let $X_1, X_2, \ldots, X_s$ be independent and identically distributed random variables of expectation $\mu$, and such that $X_i \in [0, C]$. Then, if $\widehat{\mu} = \frac{1}{s}\sum_{i=1}^{s} X_i$, we have for every $t > 0$ that
     \[
        \Pr \left[ \big|\widehat{\mu} - \mu \big| \geq t \right] \leq 2\exp\left( \frac{-2 s t^2}{C^2 } \right).
    \]
\end{lemma}

\mytheorem*
\begin{proof}
    Let $G_\mathcal{R}$ be Rubik's cube graph, which is vertex-transitive by~\Cref{lemma:cayley}. Let $s^\star$ denote the vertex of $G_\mathcal{R}$ corresponding to the solved state.
    Then, let $X_1, \ldots, X_|S|$ be i.i.d random variables corresponding to the distance between a uniformly random vertex of $G_\mathcal{R}$ and $s^\star$.
    Using~\Cref{lemma:dist-vertex-transitivity} we have $\mathbb{E}[X_i] = \mu$, and by~\Cref{lem:human}, we have that $X_i \in [0, 205]$. Also, $\widehat{\mu} = (\sum_{i=1}^{|S|}X_i)/|S|$.
    Therefore, applying~\Cref{lem:hoeffding} with $t = 0.1$, we get 
   \[
    \Pr[\mu \geq \widehat{\mu} + 0.18] \leq 2\exp\left(\frac{-2|S|\cdot 0.1^2}{205^2}\right) < 2\exp\left(\frac{-|S|}{1\,541\,939}\right).
   \]
   As $D < 2\mu$ by~\Cref{thm:d-m-vertex-transitive}, we conclude
   \[
  \Pr[D \geq \widehat{\mu} + 0.36] \leq \Pr[2\mu \geq \widehat{\mu} + 0.36] <  2\exp\left(\frac{-|S|}{1\,541\,939}\right).\qedhere
   \]
\end{proof}

\subsection{Randomly sampling cubes}

To obtain our estimates on $\mu_{\text{close}}$ and to determine we need an effective procedure for sampling states of the Rubik's cube uniformly at random.
We briefly discuss how to sample a pair of states uniformly at random from the cube.

The \emph{parity of a permutation} is the number of inversions it has modulo 2; i.e., the number of decreasing pairs of increasing entries modulo 2. 
It is well known (see e.g.,~\cite{Cubology}) that valid cube positions can be characterized by just three constraints.
Indeed, imagine that we take out all 48 non-center pieces of the Rubik's cube and rearrange them into a new state of the Rubik's cube; 
the following theorem states when such a rearrangement is a valid state of the cube. 

\begin{theorem}[Fundamental theorem of Cubology, \cite{Cubology}]
    \label{thm:fundamental-thm}
A rearrangement of the subcubes is valid if and only if    
\begin{itemize}
    \item The permutation of the corners has the same parity as the permutation of the edge.
    \item The number of corners that are twisted clockwise equals the number of corners that are twisted counterclockwise modulo three.
    \item The number of flipped edges is even.
\end{itemize}
\end{theorem}


Interestingly, \Cref{thm:fundamental-thm} implies an efficient algorithm for sampling.
Consider the following three non-valid operations in the cube:
\begin{enumerate}
    \item Flip the edge between \rr{F} and \rr{U} (see~\Cref{fig:operations}a).
    \item Swap the corners at the intersection of \rr{F}, \rr{U}, \rr{L} and \rr{F}, \rr{U}, \rr{R} (see~\Cref{fig:operations}b).
    \item Turn clockwise the corner at the intersection of \rr{D}, \rr{U}, \rr{R} (see~\Cref{fig:operations}c).
\end{enumerate}

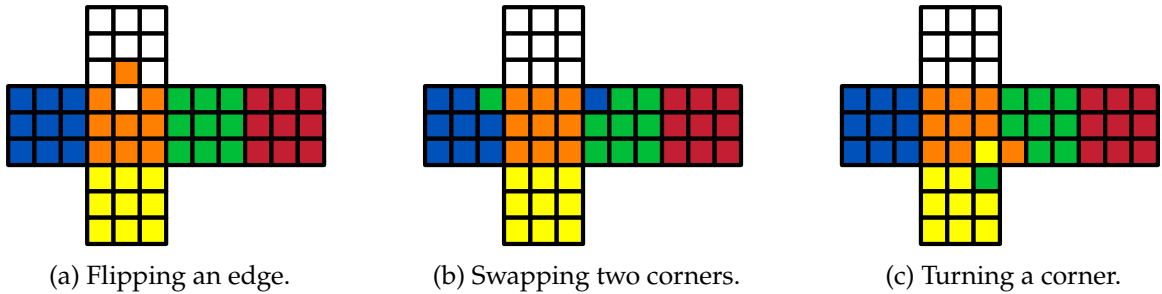
\begin{figure}[h]
    \begin{subfigure}{0.32\textwidth}
        \centering
        \begin{tikzpicture}[scale=0.35]
            \RubikFaceUp
            {W}{W}{W}
            {W}{W}{W}
            {W}{O}{W}
            \RubikFaceFront
            {O}{W}{O}
            {O}{O}{O}
            {O}{O}{O}
            \RubikFaceRight
            {G}{G}{G}
            {G}{G}{G}
            {G}{G}{G}
            \RubikFaceLeft
            {B}{B}{B}
            {B}{B}{B}
            {B}{B}{B}
            \RubikFaceBack
            {R}{R}{R}
            {R}{R}{R}
            {R}{R}{R}
            \RubikFaceDown
            {Y}{Y}{Y}
            {Y}{Y}{Y}
            {Y}{Y}{Y}
            \DrawRubikCubeF
        \end{tikzpicture}
        \label{fig:operation1}
        \caption{Flipping an edge.}
    \end{subfigure}
    \begin{subfigure}{0.32\textwidth}
        \centering
        \begin{tikzpicture}[scale=0.35]
            \RubikFaceUp
            {W}{W}{W}
            {W}{W}{W}
            {W}{W}{W}
            \RubikFaceFront
            {O}{O}{O}
            {O}{O}{O}
            {O}{O}{O}
            \RubikFaceRight
            {B}{G}{G}
            {G}{G}{G}
            {G}{G}{G}
            \RubikFaceLeft
            {B}{B}{G}
            {B}{B}{B}
            {B}{B}{B}
            \RubikFaceBack
            {R}{R}{R}
            {R}{R}{R}
            {R}{R}{R}
            \RubikFaceDown
            {Y}{Y}{Y}
            {Y}{Y}{Y}
            {Y}{Y}{Y}
            \DrawRubikCubeF
        \end{tikzpicture}
        \label{fig:operation2}
        \caption{Swapping two corners.}
    \end{subfigure}
    \begin{subfigure}{0.32\textwidth}
        \centering
        \begin{tikzpicture}[scale=0.35]
            \RubikFaceUp
            {W}{W}{W}
            {W}{W}{W}
            {W}{W}{W}
            \RubikFaceFront
            {O}{O}{O}
            {O}{O}{O}
            {O}{O}{Y}
            \RubikFaceRight
            {G}{G}{G}
            {G}{G}{G}
            {O}{G}{G}
            \RubikFaceLeft
            {B}{B}{B}
            {B}{B}{B}
            {B}{B}{B}
            \RubikFaceBack
            {R}{R}{R}
            {R}{R}{R}
            {R}{R}{R}
            \RubikFaceDown
            {Y}{Y}{G}
            {Y}{Y}{Y}
            {Y}{Y}{Y}
            \DrawRubikCubeF
        \end{tikzpicture}
        \label{fig:operation3}
        \caption{Turning a corner.}
    \end{subfigure}
    \caption{Three invalid operations.}
    \label{fig:operations}
\end{figure}

Performing these three operations leads to 12 cube configurations (there are 2 choices for flipping the edge, 2 choices for swapping the corners, and 3 choices for turning the corner).
\Cref{thm:fundamental-thm} implies that out of these 12 configurations, exactly one of them is valid, \emph{no matter the state the rest of the cube is in}.
Thus, the algorithm that reassembles the cube at random and rejects invalid states by \Cref{thm:fundamental-thm}, takes (in expectation), 12 reassembles to uniformly sample a state of the cube.
Furthermore, the algorithm that reassembles the cube at random and fixes an invalid state by performing the three operations of \Cref{fig:operations} takes only 1 re-assembly of the cube to sample uniformly.

Finally, note that, by vertex-transitivity, sampling a random pair in the cube has the same distance distribution as sampling one state of the cube and giving the distance to a \emph{fixed} state.
We will use the solved state as the fixed state in our computations and experiments.

\subsection{Experimental results}

Using the aforementioned methodology, we sampled 500\,000 uniformly random states of the Rubik's cube. Out of these, no state required more than $20$ moves to be solved, and the empirical mean distance obtained was $18.3189$. The process took under $5$ hours on a personal computer (MacBook Pro M3, 36 GB of RAM, $16$ cores). A histogram is displayed in~\Cref{fig:hist}. Our experiments can be found in~\url{https://anonymous.4open.science/r/RubikDemiGodSOSA-E3D5/README}.

\begin{figure}[h]
    \centering
    \begin{tikzpicture}
        \begin{axis}[
            nodes near coords,
            bar width=2.5em,
            nodes near coords style={anchor=south},
            width=\textwidth,
            ybar,
            ylabel = {Number of samples},
            xlabel = {Moves needed to solve by an optimized Kociemba's algorithm},
            xmin=10.5,
            ymin=0,
            xmax=20.5]
            \addplot coordinates { 
            (11, 1)
            (12, 0)
            (13, 7   )
            (14, 67   )
            (15, 956  )
            (16, 7896)
            (17, 53613 )
            (18, 215575)
            (19, 211746)
            (20, 10139)
            };
        \end{axis}
    \end{tikzpicture}
    \caption{Histogram showing 500\,000 samples. 
    This plot aggregates how many moves the samples needed to be solved. 
    No more than 20 moves were needed, and the empirical mean was 18.3189.}\label{fig:hist}
\end{figure}

\subsection{Obtaining the Demigod's number} 


While~\Cref{thm:main-1} already provides us with a way of obtaining an upper bound for the diameter, its sample complexity is not very practical. To obain a probability of error of $1\%$, we could 
set 
\[
2\exp(-|S|/1\,541\,939) \leq 0.01 \implies |S| \geq 1\,541\,939 \log(200) \approx 3 \times 10^8,
\]
from where running a solver on each sample for $0.2$ seconds, on $8$ parallel threads, would take over a year. Therefore, we now apply a simple method to reduce the sample complexity.   
We begin by showing that it is highly unlikely that many states of the cube are~\emph{``far apart''}. To this end, we say that a pair of states of the cube is \emph{``far apart''} if their distance is 21 or more, otherwise, we say that the states are \emph{``close''}.
We then consider the following statement:
\begin{equation}
    \phi := \text{``at least 0.03\% of the pairs of states of the cube are far apart.''}
    \label{eq:prop-far-apart} 
\end{equation}

If $\phi$ is true, then by sampling enough pairs of states at random we would expect to observe a pair of states that are far apart.
We formalize this idea in the following lemma whose proof is straightforward.

\begin{lemma}
    \label{lem:evidence}
    If the statement $\phi$ (from \eqref{eq:prop-far-apart}) is true, then the probability of sampling $s$ pairs of cube states uniformly at random and observing 0 pair of states that are far apart is at most $(1 - 0.0003)^s$.
\end{lemma}


We now consider the mean distance between close vertices, $\mu_{\text{close}}$; that is, 
\[\mu_{\text{close}} :=  \sum_{\substack{u,\, v \, \in \customBinom{V}{2}\\ u,v \text{ close}}} d(u,v) . \]
Moreover, we consider the empirical mean between close vertices, which we denote by  $\widehat{\mu_{\text{close}}}$; that is, we sample pairs at random, compute the distance, discard the results whenever the pairs where far apart, and average the results.
Then, we can use~\Cref{lem:hoeffding} with $C=20$ to state that $\mu_{\text{close}}$ and $\widehat{\mu_{\text{close}}}$ are close to each other with high probability. 


\begin{lemma}
    \label{lem:concentration}
   Let $\widehat{\mu_{\text{close}}}(s)$ be the empirical mean distance over $s$ samples. Then, the probability that $|\widehat{\mu_{\text{close}}}(s) - \mu_{\text{close}}| \geq 0.1$ is bounded from above by $2\exp ( -0.00005\cdot s)$.
\end{lemma}

By~\Cref{lem:evidence}, if we assume statement $\phi$, then the probability of observing no pair of states that are far apart when sampling 500\,000 pairs of states is at most 
\[ 
(1 - 0.0003)^{500\,000} < 7.02 \times 10^{-66}.
\]
However, that is exactly what we observed, and it can be easily certified by exhibiting the sequences of moves that transform one state into the other (included in our supplementary material). Therefore, we have very significant empirical evidence (i.e., $\geq 1 - 10^{-65}$ confidence interval) for the statement $\phi$ being false. Equivalently, we have empirical evidence for:
\begin{equation}
    \text{``at most 0.03\% of the pairs of states of the cube are far apart.''}
    \label{eq:prop-no-far-apart} 
\end{equation}

Furthermore, the empirical mean observed with $500\,000$ samples was $18.3189$. 

Since 
\(
2\exp( -0.00005 \cdot 500\,000) \approx 2.777 \times 10^{-11},
\) 
by~\Cref{lem:concentration}, we have \emph{significant empirical evidence} that
\begin{equation}
    \text{$\mu_{\text{close}} \leq 18.3189+ 0.1 = 18.4189$.}
    \label{eq:mean-close} 
\end{equation}

Combining facts \eqref{eq:prop-no-far-apart} and \eqref{eq:mean-close} with~\Cref{lem:human}, we have significant empirical evidence that 
\begin{equation}
        \mu \leq \mu_{\text{close} } + 0.03\% \cdot 205 \leq 18.4189+0.0615 =18.4804.
        \label{eq:final-eq} 
\end{equation}

In summary, the probability of seeing our empirical observations if the Rubik's cube diameter were to be greater than $36$ is no more than 
\[
2.777 \times 10^{-11} + 7.02 \times 10^{-66} < 10^{-10},
\]
by a union bound, and thus we have shown how to obtain a high degree of confidence with relatively few samples.


\section{Discussion}\label{sec:conclusions}
We have presented a novel approach to bounding the number of moves required to solve any state of the $3\times 3 \times 3$ Rubik's cube, which relies on two simple aspects of the cube: (i) its vertex-transitivity, and (ii) our ability to efficiently sample uniformly random states from it. Because these two properties extend to a variety of combinatorial puzzles (both to puzzles in the Rubik's family, such as the~\emph{Piraminx} or the~\emph{Megaminx}, as well as unrelated puzzles like the 15-puzzle on a torus).
Moreover, while God's number is known for the $3 \times 3 \times 3$ cube, it remains widely open for larger puzzles~\cite{speedsolvingGodapossAlgorithm} (with a gap larger than by a factor of $2$ between lower bound and upper bound already for the $5 \times 5 \times 5$ cube), where our approach could be useful provided a good algorithm without requiring theoretical guarantees on it. In terms of related work, So Hirata has recently released two interesting papers concerning the diameter of Rubik's puzzles~\cite{hirata2024graphtheoreticalestimatesdiametersrubiks,hirata2024probabilisticestimatesdiametersrubiks}, which attack the problem from different angles; either considering the girth of the cube's graph or using estimations for its branching factor. On a more theoretical line of work, Demaine et al. proved that computing the diameter of an $N \times N \times N$ Rubik's cube is NP-hard~\cite{DBLP:conf/stacs/DemaineER18}\footnote{More in general, computing the diameter of a Cayley graph is NP-hard given a group presentation~\cite{caiRoutingCirculantGraphs1999}.}, and that the diameter of the $N \times N \times N$ cube is $\Theta\left(\frac{N^2}{\log N}\right)$~\cite{DBLP:conf/esa/DemaineDELW11}.


A particular characteristic of our approach is that our result has an intermediate epistemic status between a theorem and a heuristic, albeit in our opinion much closer to the former. The situation, more in general, is closely related to the question of how much power randomness gives to computation, for which Avi Widgerson recently received a Turing Award~\cite{aviTuring}. In the context of mathematical results, such a question may be phrased as follows:
\begin{center}
\emph{Are there properties of finite mathematical objects that can only be certified efficiently to a high degree of confidence by probabilistic algorithms, but that we never can be certain of through a short proof?}
\end{center}
For instance, consider primality testing; whether a given number is prime or not is a fully deterministic fact, in the same way as the average distance of the Rubik's cube graph is. However, in order to practically obtain knowledge of such deterministic facts, we leverage the computational benefit of randomness, which allows us (at least in current practice), to determine facts that otherwise would be out of reach. The cost, however, is the possibility of error in the associated randomized algorithms, which forbids us from claiming to have definite proofs of the facts of interest. As usual, we can get such probability of error to be as small as we deem necessary for convincing ourselves, at a modest computational price, while keeping the curse of never reaching $100\%$ confidence.
An interesting counterpoint is to discuss whether traditional proofs equal certainty, as it is not evident that when reading traditional proofs we can reliably reach $100\%$ of confidence either. We might claim that for simple proofs like the irrationality of $\sqrt{2}$, the elementary proof of~\Cref{thm:d-m-vertex-transitive} in this paper, or even the Central Limit Theorem. However, proofs that span dozens or even hundreds of pages, covering a multitude of cases, and including non-trivial calculations, are much more delicate from a trust perspective. For instance, the proof of Kepler's conjecture by Thomas Hales took years before reviewers, from the prestigious Annals of Mathematics, accepted the paper while saying they were only ``99\% sure of its correctness''~\cite{Hales2005, newscientistMathematicalProofs}. For a more general discussion of the impact of computation in modern mathematics, and how our understanding of ``proofs'' can be affected by computation, we refer the interested reader to the work of Avigad~\cite{Avigad2021, avigadVarietiesMathematicalUnderstanding2023}.

Going back to our case, we have shown that, if one assumes momentarily that the diameter of the Rubik's cube graph is larger than $36$, then observing an emipirical mean distance of around $18.3$ over $500\,000$ samples, none of which required more than $20$ moves, has probability under $10^{-10}$. We encourage the readers to reproduce this computation by themselves, which should take less than a day in any modern computer. We believe this makes an extremely compelling case for the diameter of the Rubik's cube graph being at most $36$ while using a fraction of the computation required by previous approaches. Moreover, we hope that this same line of attack can be useful for analyzing other puzzles or graphs.

\paragraph{Acknowledgments.} The second author thanks Ilan Newman for a discussion that helped simplify the proof of~\Cref{thm:fundamental-thm}, and Jeremy Avigad for references on ``plausibility''. We thank Elias Frantar for making his efficient Rubik's cube solver publicly available, which facilitated this project.

\bibliographystyle{plain}
\bibliography{references}

\appendix 

\section{The Beginner's Method and the ``Human's Number''}\label{sec:beginner}
In order to make this article self-contained, we provide a brief overview of the \emph{``beginner's method''} to solve the Rubik's Cube\footnote{We encourage, nonetheless, the interested reader to look into the many YouTube videos (e.g.,~\cite{youtubeLearnSolve}) that guide the process. }, from which we have the following:

\begin{lemma}[Human's Number]\label{lem:human}
    Any position of the Rubik's cube can be solved in at most 205 moves.
\end{lemma}

In a nutshell, the beginner's method consists of solving the Rubik's Cube by \emph{``layers''}, as opposed to by faces. Before we begin with its exposition, however, it is worth establishing some notation for the different Rubik's cube moves. We will use ``Singmaster'' notation, credited to British mathematician David Singmaster. To specify a turn on a face, we use the first letter of the face's name: \rr{R} for the right face, \rr{L} for the left face, \rr{U} for the upper face, \rr{D} for the down face, \rr{F} for the front face, and \rr{B} for the back face.
 If the face is to be rotated by $90^\circ$ clockwise, we add no suffix, e.g., \rr{R} means a clockwise rotation by $90^\circ$ of the right face. For a counterclockwise rotation by $90^\circ$, we add a prime symbol, e.g., \rr{Dp} means a counterclockwise rotation by $90^\circ$ of the down face. For a $180^\circ$ rotation, we add a $2$ after the letter, e.g., \rr{F}2 means a $180^\circ$ rotation of the front face. A proper formalization of what a \emph{``move''} actually is can be found in~\Cref{sec:prelim}, where we view the Rubik's cube as a group.
As an example to check our understanding of the notation, the non-commutativity of the Rubik's cube group is evidenced in~\Cref{fig:non-commutativity}.

\begin{figure}
    \newcommand{\rd}{[rd],R,Dp}%
    \newcommand{\dr}{[dr],Dp,R}%
    \begin{subfigure}{0.45\textwidth}
        \centering
        \includegraphics{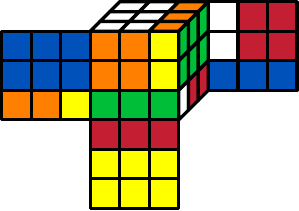}
        \caption{Result of \rr{R}\rr{Dp} from the solved state.}
    \end{subfigure}
    \begin{subfigure}{0.45\textwidth}
        \centering
        \includegraphics{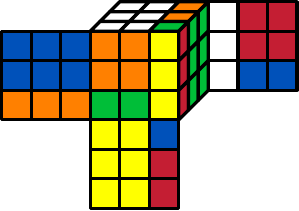}
        \caption{Result of \rr{Dp}\rr{R} from the solved state.}
    \end{subfigure}
    \caption{Illustration of the non-commutativity of the Rubik's cube.}\label{fig:non-commutativity}
\end{figure}

\paragraph{Step 1: The white cross}
    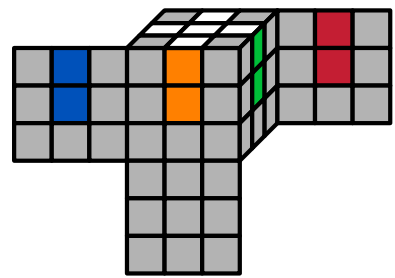
\begin{wrapfigure}{r}{0.4\textwidth}
        \begin{tikzpicture}[scale=0.5]
            \RubikFaceUp
            {X}{W}{X}
            {W}{W}{W}
            {X}{W}{X}
            \RubikFaceFront
            {X}{O}{X}
            {X}{O}{X}
            {X}{X}{X}
            \RubikFaceRight
            {X}{G}{X}
            {X}{G}{X}
            {X}{X}{X}
            \RubikFaceLeft
            {X}{B}{X}
            {X}{B}{X}
            {X}{X}{X}
            \RubikFaceBack
            {X}{R}{X}
            {X}{R}{X}
            {X}{X}{X}
            \DrawRubikCubeSF
        \end{tikzpicture}
        \caption{The white cross is solved.}\label{fig:cross}
    \end{wrapfigure}
        The first step consists of creating a \emph{``cross''} on one face of the cube, which we will assume to be white without loss of generality. To achieve this, one must move every white ``edge'' (i.e., a piece with two colors) to the correct position. That is, e.g., the white-orange edge must be placed so that its white sticker is adjacent to the white center and its green sticker is adjacent to the orange center.
        We illustrate the result of this step in~\Cref{fig:cross}, and a conservative bound is that this can always be achieved in 20 moves, as each of the 4 white edges to place can always be placed in at most 5 moves or fewer. This step can be done intuitively, that is, without memorizing any particular algorithm\footnote{In the Rubik's cube literature, a move sequence with a concrete purpose (e.g., permuting 3 corner pieces) is traditionally called an ``algorithm''.}.

\paragraph{Step 2: White corners}

    \begin{figure}
        \centering
        \begin{tikzpicture}[scale=0.5]
            \RubikFaceUp
            {W}{W}{W}
            {W}{W}{W}
            {W}{W}{W}
            \RubikFaceFront
            {O}{O}{O}
            {X}{O}{X}
            {X}{X}{X}
            \RubikFaceRight
            {G}{G}{G}
            {X}{G}{X}
            {X}{X}{X}
            \RubikFaceLeft
            {B}{B}{B}
            {X}{B}{X}
            {X}{X}{X}
            \RubikFaceBack
            {R}{R}{R}
            {X}{R}{X}
            {X}{X}{X}
            \RubikFaceDown
            {X}{X}{X}
            {X}{Y}{X}
            {X}{X}{X}
            \DrawRubikCubeSF
        \end{tikzpicture}
        \caption{The white corners are solved, and thus the first layer is completed.}\label{fig:corners}
    \end{figure}

    The second step consists of placing the white corners in their correct position, one by one. To place a white corner, one can first bring it to the opposite layer (i.e., the bottom layer, whose center is yellow), and then proceed according to a handful of cases, one of which is illustrated in~\Cref{fig:corner-steps}.  
    The result of this step is illustrated in~\Cref{fig:corners}.  Conservatively, this step can always be achieved in 15 moves per corner, and 60 in total. 
    This accounts for the cases when a white corner is in the correct location but oriented incorrectly, in which case a non-white corner can be placed in that spot, thus allowing the white corner to be placed in the correct orientation afterward. 

\begin{figure}
        \begin{subfigure}{0.292\textwidth}
            \begin{tikzpicture}[scale=0.4]
                \RubikFaceUp
                {X}{W}{X}
                {W}{W}{W}
                {X}{W}{X}
                \RubikFaceFront
                {X}{O}{X}
                {X}{O}{X}
                {X}{X}{W}
                \RubikFaceRight
                {X}{G}{X}
                {X}{G}{X}
                {G}{X}{X}
                \RubikFaceLeft
                {X}{B}{X}
                {X}{B}{X}
                {X}{X}{X}
                \RubikFaceBack
                {X}{R}{X}
                {X}{R}{X}
                {X}{X}{X}
                \RubikFaceDown
                {X}{X}{O}
                {X}{Y}{X}
                {X}{X}{X}
                \DrawRubikCubeSF
                \node at (7.3, -1.7) {\Rubik{Dp}};
            \end{tikzpicture}
        \end{subfigure}
        \def\cubeHeight{0.9}
        \begin{subfigure}{0.17\textwidth}
            \begin{tikzpicture}[scale=0.4]
                \RubikFaceUp
                {X}{W}{X}
                {W}{W}{W}
                {X}{W}{X}
                \RubikFaceFront
                {X}{O}{X}
                {X}{O}{X}
                {G}{X}{X}
                \RubikFaceRight
                {X}{G}{X}
                {X}{G}{X}
                {X}{X}{X}
                \RubikFaceLeft
                {X}{B}{X}
                {X}{B}{X}
                {X}{X}{X}
                \RubikFaceBack
                {X}{R}{X}
                {X}{R}{X}
                {X}{X}{X}
                \RubikFaceDown
                {X}{X}{O}
                {X}{Y}{X}
                {X}{X}{X}
                \DrawRubikCubeRU
                \node at (5.3, \cubeHeight) {\Rubik{Rp}};
            \end{tikzpicture}
        \end{subfigure}
        \begin{subfigure}{0.17\textwidth}
            \begin{tikzpicture}[scale=0.4]
                \RubikFaceUp
                {X}{W}{X}
                {W}{W}{X}
                {X}{W}{X}
                \RubikFaceFront
                {X}{O}{X}
                {X}{O}{W}
                {X}{X}{X}
                \RubikFaceRight
                {X}{X}{X}
                {G}{G}{X}
                {X}{X}{X}
                \RubikFaceLeft
                {X}{B}{X}
                {X}{B}{X}
                {X}{X}{X}
                \RubikFaceBack
                {X}{R}{X}
                {X}{R}{X}
                {X}{X}{X}
                \RubikFaceDown
                {X}{X}{O}
                {X}{Y}{X}
                {X}{X}{X}
                \DrawRubikCubeRU
                \node at (5.3, \cubeHeight) {\Rubik{D}};
            \end{tikzpicture}
        \end{subfigure}
        \begin{subfigure}{0.17\textwidth}
            \begin{tikzpicture}[scale=0.4]
                \RubikFaceUp
                {X}{W}{X}
                {W}{W}{X}
                {X}{W}{X}
                \RubikFaceFront
                {X}{O}{X}
                {X}{O}{W}
                {X}{X}{W}
                \RubikFaceRight
                {X}{X}{X}
                {G}{G}{X}
                {G}{X}{X}
                \RubikFaceLeft
                {X}{B}{X}
                {X}{B}{X}
                {X}{X}{X}
                \RubikFaceBack
                {X}{R}{X}
                {X}{R}{X}
                {X}{X}{X}
                \RubikFaceDown
                {X}{X}{O}
                {X}{Y}{X}
                {X}{X}{X}
                \DrawRubikCubeRU
                \node at (5.3, \cubeHeight) {\Rubik{R}};
            \end{tikzpicture}
            \end{subfigure}
            \begin{subfigure}{0.17\textwidth}
                \begin{tikzpicture}[scale=0.4]
                    \RubikFaceUp
                    {X}{W}{X}
                    {W}{W}{W}
                    {X}{W}{W}
                    \RubikFaceFront
                    {X}{O}{O}
                    {X}{O}{X}
                    {X}{X}{X}
                    \RubikFaceRight
                    {G}{G}{X}
                    {X}{G}{X}
                    {X}{X}{X}
                    \RubikFaceLeft
                    {X}{B}{X}
                    {X}{B}{X}
                    {X}{X}{X}
                    \RubikFaceBack
                    {X}{R}{X}
                    {X}{R}{X}
                    {X}{X}{X}
                    \RubikFaceDown
                    {X}{X}{O}
                    {X}{Y}{X}
                    {X}{X}{X}
                    \DrawRubikCubeRU
                    \node at (5.3, \cubeHeight) {\phantom{\Rubik{D}}};
                \end{tikzpicture}
             \end{subfigure}
        \caption{Illustration of one of the cases for placing a corner in the first layer (Step 2 of the beginner's method), through the move sequence \rr{Dp}\rr{Rp}\rr{D}\rr{R}.}\label{fig:corner-steps}
    \end{figure}


\paragraph{Step 3: Edges of the second layer}
    \begin{wrapfigure}{r}{0.4\textwidth}
        \centering
        \begin{tikzpicture}[scale=0.5]
            \RubikFaceUp
            {W}{W}{W}
            {W}{W}{W}
            {W}{W}{W}
            \RubikFaceFront
            {O}{O}{O}
            {O}{O}{O}
            {X}{X}{X}
            \RubikFaceRight
            {G}{G}{G}
            {G}{G}{G}
            {X}{X}{X}
            \RubikFaceLeft
            {B}{B}{B}
            {B}{B}{B}
            {X}{X}{X}
            \RubikFaceBack
            {R}{R}{R}
            {R}{R}{R}
            {X}{X}{X}
            \RubikFaceDown
            {X}{X}{X}
            {X}{Y}{X}
            {X}{X}{X}
            \DrawRubikCubeSF
        \end{tikzpicture}
        \caption{The first two layers are solved.}\label{fig:f2l}
    \end{wrapfigure}
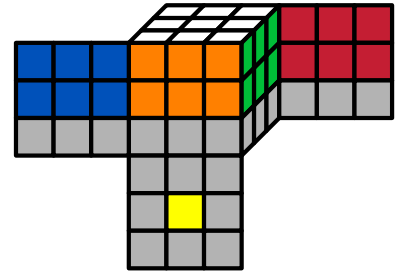
    The third step consists of placing the edges of the second layer in their final position, as illustrated in~\Cref{fig:f2l}.
    For instance, the orange-green edge must be placed so that its orange sticker is adjacent to the orange center and its green sticker is adjacent to the green center. 
    The main algorithm to solve this step is illustrated in~\Cref{fig:f2l-algo}. A conservative bound, again due to cases in which a misoriented edge must be first replaced before placing it in the correct orientation, is 20 moves per edge.

\begin{figure}

\includegraphics{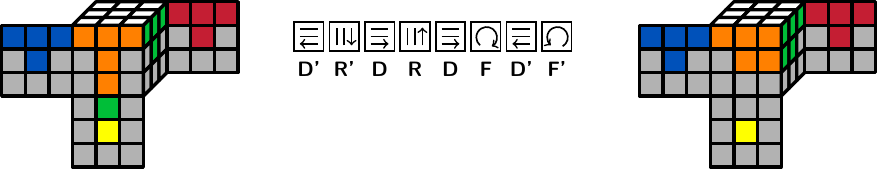}
\caption{Illustration of the algorithm to solve the edges of the second layer (Step 3 of the beginner's method).}\label{fig:f2l-algo}
\end{figure}

\paragraph{Step 4: The yellow cross}
We now turn our attention to the yellow face. The goal of this step is to solve the orientation of the yellow cross, that is, to make all yellow edges have their yellow sticker adjacent to the yellow center, as depicted in~\Cref{fig:yellow-cross-solved}.
We may face three different scenarios in this step (if it is not already solved), as illustrated in~\Cref{fig:yellow-cross}. These can all be solved by the same algorithm, potentially repeated according to which of the three non-solved cases we encounter. The algorithm is simply: \rr{F}\rr{R}\rr{U}\rr{Rp}\rr{Up}\rr{Fp}. Applying it from case~\ref{fig:yellow-cross-1} leads to case~\ref{fig:yellow-cross-2}, and applying it again leads to case~\ref{fig:yellow-cross-3}, from where a last application solves the yellow cross. That way, we need at most 3 applications, leading to a conservative bound of 18 moves for this step. 

\begin{figure}
\begin{subfigure}{0.3\textwidth}
    \centering
    \begin{tikzpicture}[scale=0.4]
            \RubikCubeGreyAll
    \RubikFaceUp XXX
    XYX
    XXX
    \RubikFaceFront XYX XXXXXX
    \RubikFaceRight XYX XXXXXX
    \RubikFaceBack XYX XXXXXX
    \RubikFaceLeft XYX XXXXXX
    \DrawRubikFaceUpSide
    \end{tikzpicture}
    \caption{No edges oriented correctly.}\label{fig:yellow-cross-1}
\end{subfigure}
\begin{subfigure}{0.3\textwidth}
    \centering
    \begin{tikzpicture}[scale=0.4]
            \RubikCubeGreyAll
    \RubikFaceUp XYX
    YYX
    XXX
    \RubikFaceFront XYX XXXXXX
    \RubikFaceRight XYX XXXXXX
    \RubikFaceBack XXX XXXXXX
    \RubikFaceLeft XXX XXXXXX
    \DrawRubikFaceUpSide
    \end{tikzpicture}
    \caption{Non-opposite edges oriented correctly.}\label{fig:yellow-cross-2}
\end{subfigure}
\begin{subfigure}{0.3\textwidth}
    \centering
    \begin{tikzpicture}[scale=0.4]
            \RubikCubeGreyAll
    \RubikFaceUp XXX
    YYY
    XXX
    \RubikFaceFront XYX XXXXXX
    \RubikFaceRight XXX XXXXXX
    \RubikFaceBack XYX XXXXXX
    \RubikFaceLeft XXX XXXXXX
    \DrawRubikFaceUpSide
    \end{tikzpicture}
    \caption{Opposite edges oriented correctly.}\label{fig:yellow-cross-3}
\end{subfigure}

\begin{subfigure}{\textwidth}
    \centering
    \begin{tikzpicture}[scale=0.35]
    \RubikFaceUp
    {W}{W}{W}
    {W}{W}{W}
    {W}{W}{W}
    \RubikFaceFront
    {O}{O}{O}
    {O}{O}{O}
    {O}{O}{Y}
    \RubikFaceRight
    {G}{G}{G}
    {G}{G}{G}
    {R}{B}{B}
    \RubikFaceLeft
    {B}{B}{B}
    {B}{B}{B}
    {G}{G}{Y}
    \RubikFaceBack
    {R}{R}{R}
    {R}{R}{R}
    {O}{R}{Y}
    \RubikFaceDown
    {Y}{Y}{R}
    {Y}{Y}{Y}
    {B}{Y}{G}
    \DrawRubikCubeF
    \end{tikzpicture}
    \caption{A possible state after solving the yellow cross.}\label{fig:yellow-cross-solved}
\end{subfigure}
\caption{Illustration of Step 4 of the beginner's method.}\label{fig:yellow-cross}
\end{figure}
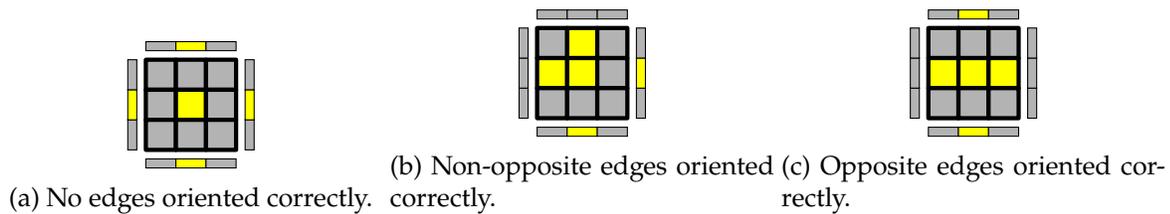
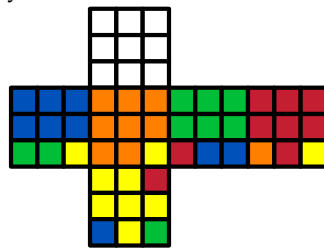


\paragraph{Step 5: Permuting yellow edges}

    Now, we permute the yellow edges so that each of them gets to its desired position. This step can be solved by repeated application of a single algorithm, that induces a $3$-cycle of the yellow edges, as illustrated in~\Cref{fig:cycle-yellow-edges,fig:yellow-edges}. As this algorithm is applied at most $3$ times, we have a conservative bound of 21 moves for this step.
\begin{figure}
    \includegraphics{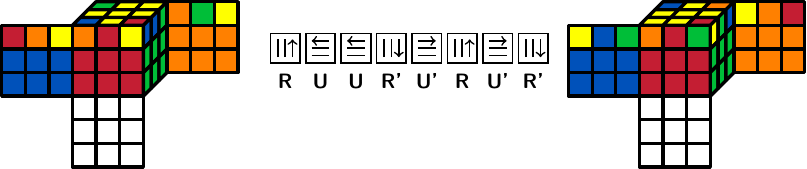}


    \caption{Illustration of a case for Step 5 of the beginner's method.}\label{fig:yellow-edges}
\end{figure}

\paragraph{Step 6: Permuting yellow corners}
\begin{figure}
\begin{subfigure}{0.45\textwidth}
        \centering
        \begin{tikzpicture}[scale=0.75]
         \DrawRubikFaceUp 
            \coordinate (A) at (2.5, 0.5);
            \coordinate (B) at (2.5, 2.5);
            \coordinate (C) at (0.5, 2.5);
            \node[draw,circle, fill=blue, inner sep=0.05cm] (a) at (A) {};
            \node[draw,circle, fill=blue, inner sep=0.05cm] (b) at (B) {};
            \node[draw, circle, fill=blue, inner sep=0.05cm] (c) at (C) {};
            \draw[->,  thick,color=yellow] (b)-- (c);
            \draw[->,  thick,color=yellow] (c)-- (a);
            \draw[->,  thick,color=yellow] (a)-- (b);
        \end{tikzpicture}
        \caption{$3$-cycle permutation of corners induced by \rr{R}\rr{Up}\rr{Lp}\rr{U}\rr{Rp}\rr{Up}\rr{L}\rr{U}.}\label{fig:cycle-yellow-corners}
    \end{subfigure}
    \hfill
    \begin{subfigure}{0.45\textwidth}
        \centering
        \begin{tikzpicture}[scale=0.75]
         \DrawRubikFaceUp 
            \coordinate (A) at (1.5, 2.5);
            \coordinate (B) at (2.5, 1.5);
            \coordinate (C) at (0.5, 1.5);
            \node[draw,circle, fill=blue, inner sep=0.05cm] (a) at (A) {};
            \node[draw,circle, fill=blue, inner sep=0.05cm] (b) at (B) {};
            \node[draw, circle, fill=blue, inner sep=0.05cm] (c) at (C) {};
            \draw[->,  thick,color=yellow] (b)-- (c);
            \draw[->,  thick,color=yellow] (c)-- (a);
            \draw[->,  thick,color=yellow] (a)-- (b);
        \end{tikzpicture}
        \caption{$3$-cycle permutation of edges induced by \rr{R}\rr{U}2\rr{Rp}\rr{Up}\rr{R}\rr{Up}\rr{Rp}.}\label{fig:cycle-yellow-edges}
    \end{subfigure}
    \caption{Illustration of the $3$-cycle algorithms for permuting yellow corners and edges, corresponding to Steps 5 and 6 of the beginner's method.}
\end{figure}
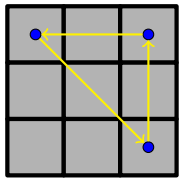
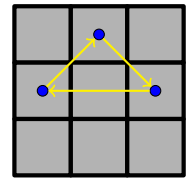

    This step is analogous to the previous one but over the corners; we permute the yellow corners so that each of them gets to its desired position. This step can also be solved by repeated application of a single algorithm, that induces a $3$-cycle of the yellow corners, as illustrated in~\Cref{fig:cycle-yellow-corners,fig:yellow-corners}. This algorithm is applied at most $3$ times, leading to a conservative bound of 24 moves for this step.
\begin{figure}
    \includegraphics{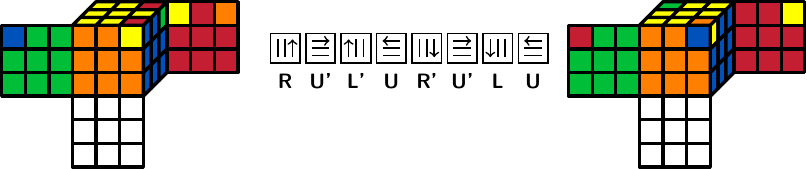}


    \caption{Illustration of a case for Step 6 of the beginner's method.}\label{fig:yellow-corners}
\end{figure}

\paragraph{Step 7: Orienting yellow corners}
    The last step consists of orienting the yellow corners, which again can be achieved by repeated applications of a single algorithm that changes the orientation of two adjacent corners (illustrated in~\Cref{fig:final-step}):
    \begin{center}
         \rr{R}\rr{U}2\rr{Rp}\rr{Up}\rr{R}\rr{Up}\rr{Rp}\rr{Lp}\rr{U}2\rr{L}\rr{U}\rr{Lp}\rr{U}\rr{L}.
    \end{center} 
    This algorithm needs to be applied at most $3$ times, leading to a conservative bound of 42 moves for this step.

    \begin{figure}
        \includegraphics{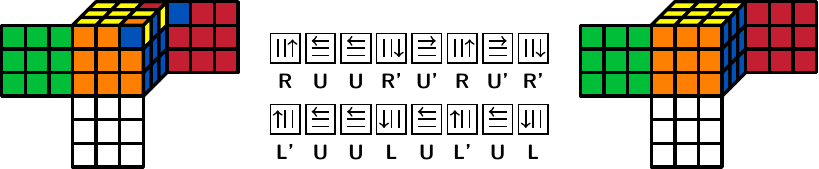}
    

        \caption{Illustration of a case for the final step of the beginner's method, using the move sequence: \rr{R}\rr{U}2\rr{Rp}\rr{Up}\rr{R}\rr{Up}\rr{Rp}\rr{Lp}\rr{U}2\rr{L}\rr{U}\rr{Lp}\rr{U}\rr{L}.}\label{fig:final-step}
    \end{figure}

The following table summarizes the ``proof'' of~\Cref{lem:human}:
\begin{table}[H]
    \centering
    \begin{tabular}{rr}
        \toprule
        Step & Moves\\ \midrule
        White cross & 20  \\
        White corners & 60  \\
        Edges of the second layer & 80  \\
        Yellow cross & 18 \\
        Permuting yellow edges & 21  \\
        Permuting yellow corners & 24  \\
        Orienting yellow corners & 42  \\ \midrule
        Total & 205\\
        \bottomrule
    \end{tabular}
\end{table}

\end{document}